\documentclass[11pt]{amsart}%
\usepackage{amsmath}
\usepackage{graphicx}
\usepackage{amsfonts}
\usepackage{amssymb}%
\providecommand{\U}[1]{\protect\rule{.1in}{.1in}}
\providecommand{\U}[1]{\protect\rule{.1in}{.1in}}
\providecommand{\U}[1]{\protect\rule{.1in}{.1in}}
\providecommand{\U}[1]{\protect\rule{.1in}{.1in}}
\providecommand{\U}[1]{\protect\rule{.1in}{.1in}}
\newtheorem{theorem}{Theorem}[section]
\theoremstyle{plain}

\newtheorem{definition}{Definition}

\newtheorem{lemma}{Lemma}[section]

\numberwithin{equation}{section}
\newcommand*\diff{\mathop{}\!\mathrm{d}}
\DeclareMathOperator{\tr}{tr}
\DeclareMathOperator{\Det}{det}
\DeclareMathOperator{\diag}{diag}

\begin{document}
\title[Rigidity of Solutions to Hessian Quotient Flows]
{Rigidity of entire convex self-shrinking solutions to Hessian quotient flows
}

\author{Wenlong Wang}

\address{School of Mathematical Sciences\\
Peking University\\
Haidian, Beijing, 100871\\
P.R.China }
\email{wwlpkumath@yahoo.com}         


\date{Received: date / Accepted: date}

\maketitle

\begin{abstract}
We prove that all entire smooth strictly convex self-shrinking solutions on $\mathbb{R}^n$ to the Hessian quotient flows must be quadratic. This generalizes the rigidity theorem for entire self-shrinking solutions to the Lagrangian mean curvature flow in pseudo-Euclidean space due to Ding-Xin \cite{DX}. Moreover, we show that our argument works for a larger class of equations. In particular, we obtain rigidity results for entire self-shrinking solutions on $\mathbb{C}^n$ to the K\"{a}hler-Ricci flow under certain conditions.
\end{abstract}

\section{Introduction}
For a $n$-dimensional symmetric matrix $\mathcal{B}$, let $\lambda=(\lambda_1,...,\lambda_n)$ denote the eigenvalues of $\mathcal{B}$. Let $\sigma_l(\mathcal{B})$ be the $l$-th elementary symmetric polynomial of $\lambda$ given by
\begin{eqnarray}
    &&\sigma_l\left(\mathcal{B}\right)=\sum_{i_1<...<i_l}\lambda_{i_1}\cdot\cdot\cdot\lambda_{i_l}\,\,\,\,\mbox{$(1\leq l\leq n)$}\nonumber; \\
    &&\sigma_0\left(\mathcal{B}\right)=1. \nonumber
 \end{eqnarray}
We say $\mathcal{B}$ is $k$-positive if $\sigma_{l}(\mathcal{B})>0$ for all $0\leq l\leq k$. Let $0\leq n_2<n_1\leq n$, for any $n_1$-positive matrix $\mathcal{B}$, we define the quotient $q_{n_1,n_2}(\mathcal{B})$ by
\begin{equation}
q_{n_1,n_2}\left(\mathcal{B}\right)=\frac{\sigma_{n_1}\left(\mathcal{B}\right)}{\sigma_{n_2}\left(\mathcal{B}\right)}. \nonumber
\end{equation}

In the present paper, we prove the following main theorem.
\begin{theorem}\label{HQ}
Let $u$ be an entire smooth strictly convex solution on $\mathbb{R}^{n}$  to the Hessian quotient equation 
\begin{equation}\label{HQeq}
\ln q_{n_1,n_2}\left(D^2u(x)\right)=\frac{1}{2}x\cdot Du\left(x\right)-u\left(x\right).
\end{equation}
Then $u$ is quadratic.
\end{theorem}
Any solution to \eqref{HQeq} leads to an entire self-shrinking solution 
\begin{equation}
v(x,t)=-tu\left(\frac{x}{\sqrt{-t}}\right) \nonumber
\end{equation}
to a parabolic Hessian quotient equation
\begin{equation}\label{HQflow}
v_t=\ln q_{n_1,n_2}\left(D^2v\right)
\end{equation}
on $\mathbb{R}^{n} \times (-\infty,0)$. In \cite{TW}, Trudinger and Wang used this flow under the fixed boundary condition to  study a Poincar\'{e} type inequality for Hessian integrals (see \cite{T} for the Monge-Amper\`{e} integral). In fact, \eqref{HQflow} is the negative logarithmic gradient flow of the following functional (cf. \cite{T,WX})
\begin{equation}
I_{n_1,n_2}(u)=\frac{1}{n_1+1}\int (-u)\cdot \sigma_{n_1}(D^2u)-\frac{1}{n_2+1}\int (-u)\cdot \sigma_{n_2}(D^2u).\nonumber
\end{equation}

When $n_1=n$, $n_2=0$, \eqref{HQeq} becomes the Monge-Amp\`ere equation
\begin{equation}\label{MAeq}
\ln\det D^2u(x)=\frac{1}{2}x\cdot Du(x)-u(x).
\end{equation}
Any solution to \eqref{MAeq} leads to an entire self-shrinking solution 
\begin{equation}
v(x,t)=-tu\left(\frac{x}{\sqrt{-t}}\right) \nonumber
\end{equation}
to a parabolic Monge-Amp\`ere equation
\begin{equation}
v_t=\ln\det D^2v\nonumber
\end{equation}
on $\mathbb{R}^{n} \times(-\infty,0)$ and the family of embeddings ${F}$$(x, t) = (x,Dv(x,t))$ from 
${\mathbb R}^{n}$ into ${\mathbb R}^{2n}$ solves the mean curvature flow with respect to the  
pseudo-Euclidean background metric $\diff s^2 = \sum_{i=1}^n\diff x^i\diff y^i$ on ${\mathbb R}^{2n}$ (cf. \cite{CM,H,Huisken,S}). 

Rigidity of entire smooth convex solutions to \eqref{MAeq} has been studied in \cite{CCY,DX,H,HW}. In \cite{CCY} and \cite{HW}, the authors proved that any smooth convex solution to \eqref{MAeq} must be quadratic under the condition that the Hessian is bounded below inversely quadratically. Later in \cite{DX}, Ding-Xin gave a complete improvement by dropping additional assumptions. 

The common part of the arguments in  \cite{CCY}, \cite{DX} and here is proving the constancy of a natural quantity, the phase $\phi=\ln\det D^2u$ ($\phi=\ln q_{n_1,n_2}(D^2u)$ in the Hessian quotient case). Then the homogeneity of the self-similar term on the right-hand side of the equation leads to the quadratic conclusion. The phase satisfies an elliptic equation without zeroth order term (shown below in \eqref{Lphase1}). In \cite{CCY}, using the inversely quadratic decay assumption, Chau-Chen-Yuan constructed a specific barrier function to force the supremum of the phase in $\mathbb{R}^n$ to be attained at some point. Then the strong maximum principle implies the constancy of the phase. In \cite{DX}, Ding-Xin first obtained the properness of $u$, then proved the constancy of the phase via the integral method.

Our approach is to construct a barrier function to force the supremum of the phase to be attained at some point. However, we cannot construct a specific barrier function as in \cite{CCY}, which requires the specific decay rate of the Hessian. We turn to estimate the growths of the solution $u$ and $|Du|$, then construct a non-concrete barrier function. To begin with, we establish a second order ordinary differential inequality for the spherical mean of $u$, a univariate function depending on the radius of the sphere. Then using some ODE techniques, we prove that the spherical mean of $u$ has at most a quadratic growth and the ball mean of $\Delta u$ is bounded. Combining these with the convexity of $u$, we obtain that $u$ has at most a quadratic growth, $|D u|$ has at most a linear growth and the negative part of $u$ has a sublinear growth. Having these estimates, we finally construct a suitable barrier function based on $u$ and $\phi$.

In fact, our argument for Theorem \ref{HQ} does not depend on the particular structure of \eqref{HQeq}. This enables us to generalize the rigidity result to a larger class of equations. 

Let $\mathcal{S}^n_+$ be the cone of $n$-dimensional positive-definite matrices. Let $F$ be a $C^1$ function defined on $\mathcal{S}^n_{+}$. For any $\mathcal{B}=\left(b_{ij}\right)\in\mathcal{S}^n_{+}$, define the coefficient matrix $DF$ by
\begin{equation*}
\left(DF\right)^{ij}\left(\mathcal{B}\right)=\frac{\partial F}{\partial b_{ij}}\left(\mathcal{B}\right).
\end{equation*}
\begin{theorem}\label{M}
Assume for any $\mathcal{B}\in\mathcal{S}^n_{+}$, $F$ satisfies the following conditions:
\begin{flalign}
&\,\,\mbox{\rm(i)}\,\,\,DF\left(\mathcal{B}\right)\,\,\mbox{is positive-definite;}\nonumber\\
&\,\mbox{\rm(ii)} \,\,\exp F\left(\mathcal{B}\right)\leq C\left[\left(\tr\mathcal{B}\right)^{k_1}+1\right]\,\,\,\mbox{for certain positive constants $k_1$ and $C$.}\,\,\nonumber\\
&\mbox{\rm(iii)}\,\,\Vert DF\left(\mathcal{B}\right)\cdot\mathcal{B}\Vert\leq k_2\,\,\,\mbox{for a certain  positive constant $k_2$.}\nonumber
\end{flalign}
Let $u$ be an entire smooth strictly convex solution on $\mathbb{R}^n$ to the equation
\begin{equation}\label{Meq}
F\left(D^2u\left(x\right)\right)=\frac{1}{2}x\cdot Du\left(x\right)-u\left(x\right).
\end{equation}
Then $u$ is quadratic.
\end{theorem}
Condition (i) guarantees the ellipticity of \eqref{Meq}. Conditions \rm(ii) and \rm(iii) say that \eqref{Meq} has exponential or super-exponential nonlinearity for the quadratic self-similar term on the right-hand side in a sense. We are about to show that some common operators satisfy above conditions.

Let us first verify that $\ln q_{n_1,n_2}$ satisfies these conditions. For condition \rm(i), $DF(\mathcal{B})$ is positive-definite when $\mathcal{B}$ is $n_1$-positive. Namely, equation \eqref{HQeq} is elliptic when $u$ is $n_1$-admissible (cf. \cite{CNS,TW,WX}). Since $u$ is strictly convex, it is $n_1$-admissible. We can also check condition (i) directly by diagonalizing $\mathcal{B}$ and using Newton's inequality (cf. \cite{HLP}).

 For condition (ii), also by Newton's inequality we have
\begin{equation}\label{HQ1}
q_{n_1,n_2}(\mathcal{B})\leq C(n,n_1,n_2)\left(\tr\mathcal{B}\right)^{n_1-n_2}.
\end{equation}

For condition (iii), since $q_{n_1,n_2}(\mathcal{B})$ is a homogeneous order $n_1-n_2$ function of $\mathcal{B}$, by Euler's homogeneous function theorem we have
\begin{equation}\label{HQ2}
\tr\left(D\ln q_{n_1,n_2}\left(\mathcal{B}\right)\cdot\mathcal{B}\right)=n_1-n_2.
\end{equation}
Because $\ln q_{n_1,n_2}$ is invariant under orthogonal transformations, $D\ln q_{n_1,n_2}\left(\mathcal{B}\right)$ and $\mathcal{B}$ can be diagonalized simultaneously. Thus $D\ln q_{n_1,n_2}\left(\mathcal{B}\right)$ commutes with $\mathcal{B}$. Then $D\ln q_{n_1,n_2}\left(\mathcal{B}\right)\cdot\mathcal{B}$ is positive-definite. Consequently, 
\begin{equation}\label{HQ3}
\Vert D\ln q_{n_1,n_2}\left(\mathcal{B}\right)\cdot\mathcal{B}\Vert<\tr\left(D\ln q_{n_1,n_2}\left(\mathcal{B}\right)\cdot\mathcal{B}\right)=n_1-n_2.
\end{equation}

We can verify that the operator $\tr\left(\arctan\mathcal{B}\right)$ also satisfies above three conditions. The corresponding equation
\begin{equation}\label{arctan}
\sum_{i=1}^n\arctan\lambda_i\left(x\right)=\frac{1}{2}x\cdot Du\left(x\right)-u\left(x\right)
\end{equation}
describes the potential of the self-shrinking solution $(x,Du(x))$ to the Lagrangian mean curvature flow in $\mathbb{R}^{2n}$ (cf. \cite{CCH,CCY,CM,DX,H,HW,Huisken,S}). In \cite{CCY}, 
Chau-Chen-Yuan first proved that any entire smooth solution to \eqref{arctan} on $\mathbb{R}^n$ must be quadratic.

The Hermitian counterpart of \eqref{MAeq} is the following complex Monge-Amp\`{e}re equation
\begin{equation}\label{KReq}
\ln\det\partial\bar{\partial}u\left(x\right)=\frac{1}{2}x\cdot Du\left(x\right)-u\left(x\right).
\end{equation}
Any solution to \eqref{KReq} leads to an entire self-shrinking solution 
\begin{equation}
v\left(x,t\right)  =-tu\left(\frac{x}{\sqrt{-t}}\right)\nonumber
\end{equation}
to a parabolic complex Monge-Amp\`ere equation 
\begin{equation}
v_t=\ln\det\partial\bar{\partial}v\nonumber
\end{equation}
on $\mathbb{C}^{n} \times (-\infty,0)$. Note that the above equation of $v$ is the potential equation of the K\"{a}hler-Ricci flow $\displaystyle\partial_tg_{\alpha\bar{\beta}}=-R_{\alpha\bar\beta}$. In fact, the corresponding metric $\left(u_{\alpha\bar{\beta}}\right)$ is a shrinking K\"{a}hler-Ricci (non-gradient) soliton (cf. \cite{CCY}).

Rigidity of entire solutions to \eqref{KReq} has been studied in \cite{CCY,DX,DLY,WW}. In \cite{DLY}, Drugan-Lu-Yuan proved that any complete (with respect to the corresponding K\"{a}her metric $\partial\bar\partial u$) solution has to be quadratic. In \cite{WW}, completeness assumption is removed for complex one dimensional case. Using our argument, we can obtain two new rigidity theorems which are described now.

We know $$\det\partial\bar\partial u=4^{-n}\sqrt{\det\left(D^2u+J^{\mathrm{T}}\cdot D^2u\cdot J\right)},$$
where $J$ denotes the standard complex structure of $\mathbb{R}^{2n}$ and $J^{\mathrm{T}}$ is the transpose of $J$ with $J^{\mathrm{T}}=-J$. Accordingly, the ``complex determinant'' operator $\Det_{J}$ for $\mathcal{B}\in\mathcal{S}^{2n}_+$ is defined by 
$$\Det_{J}\mathcal{B}=4^{-n}\sqrt{\det\left(\mathcal{B}-J\mathcal{B}J\right)}.$$
Let us verify that $\ln\Det_{J}$ satisfies conditions (i) and (ii). For condition (i), we have 
$$D\left(\ln\Det_{J}\mathcal{B}\right)=\left(\mathcal{B}-J\mathcal{B}J\right)^{-1}.$$ 
Since $\mathcal{B}>0$, we have $\mathcal{B}-J\mathcal{B}J>0$. Then $D\left(\ln\Det_{J}\mathcal{B}\right)$ is positive-definite.
Actually, $D(\ln\det\partial\bar\partial u)$ is a quarter of the real representation of $\displaystyle\left(\partial\bar\partial u\right)^{-1}$. Equation \eqref{KReq} is elliptic if and only if $u$ is pluri-subharmonic. Since $u$ is strictly convex, it is pluri-subharmonic. For condition (ii), by the arithmetic mean-geometric mean inequality we have
\begin{equation}\label{KR01}
\Det_{J}\mathcal{B}\leq 4^{-n}\left(\frac{1}{2n}\tr\left(\mathcal{B}-J\cdot\mathcal{B}\cdot J\right)\right)^{n}=\frac{1}{(4n)^n}\left(\tr\mathcal{B}\right)^n.
\end{equation}

Because $\Det_{J}\mathcal{B}$ is a homogeneous order $n$ function of $\mathcal{B}$, by Euler's homogeneous function theorem we have $\tr\left(D\ln\Det_{J}\mathcal{B}\cdot\mathcal{B}\right)=n$. However, $\left(\mathcal{B}-J\mathcal{B}J\right)^{-1}$ and $\mathcal{B}$ do not commute in general. So $\ln\Det_{J}$ does not satisfy condition (iii), our method is not suitable to a general convex function $u$. But if $u$ satisfies one of the following conditions, the rigidity theorem still holds. 
\begin{definition}\label{df1}
For a pluri-subharmonic function $u$ on $\mathbb{C}^n$, we say \textbf{the eigenvalues of $\partial\bar\partial u$ are comparable}, if there is a constant $\Lambda\geq 1$ such that
\begin{equation}\label{KRc1}
\mu_{\max}(x)\leq\Lambda \mu_{\min}(x)\,\,\,\,\,\mbox{for any $x\in\mathbb{C}^n$},
\end{equation}
where $\mu_{\max}(x)$ and $\mu_{\min}(x)$ are the largest and the smallest eigenvalues of $\partial\bar\partial u(x)$ respectively.
\end{definition}
\begin{definition}\label{df2}
A function $u$ on $\mathbb{C}^n$ is called \textbf{toric} if
\begin{equation*}
u\left(z^1,...,z^n\right)=u\left(e^{\sqrt{-1}t^1}z^1,...,e^{\sqrt{-1}t^n}z^n\right)\,\,\,\,\,\mbox{for any $(t^1,...,t^n)\in\mathbb{R}^n$}.
\end{equation*}
\end{definition}
\begin{theorem}\label{KR1}
Let $u$ be an entire smooth strictly convex solution on $\mathbb{C}^n$ to \eqref{KReq}. Assume the eigenvalues of $\partial\bar\partial u$ are comparable. Then $u$ is quadratic.
\end{theorem}
\begin{theorem}\label{KR2}
Let $u$ be an entire smooth convex solution on $\mathbb{C}^n$ to \eqref{KReq}.
Assume $u$ is toric. Then $u$ is quadratic.
\end{theorem}

Equation \eqref{HQeq} has a relationship with Legendre transformation (cf. \cite{HW}). Suppose that $u$ is a strictly convex solution to \eqref{HQeq}, then the Legendre transform of $u$ denoted by $u^*$ satisfies 
\begin{equation}\label{LTeq}
\ln q_{n-n_2,n-n_1}\left(D^2 u^*\right)=\frac{1}{2}x\cdot Du^*-u^*.
\end{equation}
In particular, when $n_1+n_2=n$, \eqref{HQeq} is invariant under Legendre transformation. Taking advantage of this relation, we have the following theorem.
\begin{theorem}\label{LT}
Let $\Omega\subset\mathbb{R}^n$ be a bounded convex domain. Assume $u$ is a smooth strictly convex solution to \eqref{HQeq} in $\Omega$. Then $u$ is improper.
\end{theorem}
\begin{proof}
We proceed by contradiction. Assume $u$ is proper, then $u^*$ is an entire smooth strictly convex solution on $\mathbb{R}^n$ to \eqref{LTeq}. According to Theorem \ref{HQ}, $u^*$ is quadratic. By the property of Legendre transform, $u$ is also quadratic. Since a quadratic function cannot be proper in a bounded domain, the assumption is not true. Therefore $u$ is improper.
\end{proof}

Although Theorem \ref{HQ} is a special case of Theorem \ref{M}, its proof is more original and explicit. And readers can get the proof of Theorem \ref{M} from the proof of Theorem \ref{HQ} easily with only change of symbols and constants. So we only prove Theorem \ref{HQ} in the following. For the proof of Theorem \ref{KR1} and Theorem \ref{KR2}, we skip the common part with Theorem \ref{HQ}, and only talk about the difference.

\section{Proof of Theorem \ref{HQ}}
To get appropriate estimates for the solution $u$, we establish four lemmas. In these lemmas, we prove that if $u$ satisfies the conditions of Theorem \ref{M}, then $u$ has at most a quadratic growth, $|Du|$ has at most a linear growth and $u^-$ has a sublinear growth.

In the first two lemmas, we derive a second order ordinary differential inequality for the spherical mean of $\tilde u$, where $\tilde u$ is related to $u$ by a simple linear transform (shown below in \eqref{lt}). Then we prove that the spherical mean of $\tilde u$ has at most a quadratic growth and this property is passed on to $u$.
\begin{definition}
For a $C^2$ function $h(x)$ on $\mathbb{R}^n$, define 

(i) the spherical mean of $h$ by
\begin{equation}
S_h(r)=\frac{1}{\omega_n r^{n-1}}\int_{\partial B_r(0)}h(x)\diff S_x, \nonumber
\end{equation}
and 

(ii) the ball mean of $\Delta h$ by
\begin{equation}
\Phi_h(r)=\frac{n}{\omega_n r^n}\int_{B_r(0)}\Delta h(x)\diff x,\nonumber
\end{equation}
where $\omega_n$ is the surface area of the unit sphere in $\mathbb{R}^n$.
\end{definition}

\begin{lemma}\label{lem1}
Let $h(x)$ be a $C^2$ function on $\mathbb{R}^n$ satisfying
\begin{equation}\label{lemma1}
\Delta h(x)\geq\exp\left[x\cdot Dh\left(x\right)-2h\left(x\right)\right].
\end{equation}
Then $S_h(r)$ has at most a quadratic growth and $\Phi_h(r)$ is bounded.
\end{lemma}

\begin{proof}
First of all, we derive a differential inequality for $S_h$. By definition, 
\begin{equation}
S_h(r)=\frac{1}{\omega_n}\int_{\partial B_1(0)}h(r\omega)\diff\omega. \nonumber
\end{equation}
Taking one derivative, we have
\begin{equation}\label{sm1}
S'_h(r)=\frac{1}{\omega_n}\int_{\partial B_1(0)}h_r(r\omega)\diff\omega. 
\end{equation}
Multiplying $r^{n-1}$ on both sides of \eqref{sm1} and using Stokes's formula,
\begin{equation}
\begin{split}\label{sm2}
r^{n-1}S'_h(r)&=\frac{1}{\omega_n}\int_{\partial B_1(0)}h_r(r\omega)\,r^{n-1}\diff\omega \\
&=\frac{1}{\omega_n}\int_{\partial B_r(0)}\frac{\partial h}{\partial\nu}\left(x\right)\diff S_x \\
&=\frac{1}{\omega_n}\int_{B_r(0)}\Delta h\left(x\right)\diff x.
\end{split}
\end{equation}
A differentiation of \eqref{sm2} yields
\begin{equation}
\left[r^{n-1}S'_h(r)\right]'=\frac{1}{\omega_n}\int_{\partial B_r(0)}\Delta h\left(x\right)\diff S_x.\nonumber
\end{equation}
Dividing both sides of above equation by $r^{n-1}$ and using \eqref{lemma1}, we get
\begin{align*}
S''_h(r)+\frac{n-1}{r}S_h'(r)&=\frac{1}{r^{n-1}}\left[r^{n-1}S_h'(r)\right]' \\
&=\frac{1}{\omega_n r^{n-1}}\int_{\partial B_r(0)}\Delta h\left(x\right)\diff S_x \\
&\geq\frac{1}{\omega_n r^{n-1}}\int_{\partial B_r(0)}\exp\left[x\cdot Dh(x)-2h(x)\right]\diff S_x.
\end{align*}
By Jensen's inequality we obtain
\begin{align*}
&\frac{1}{\omega_n r^{n-1}}\int_{\partial B_r(0)}\exp\left[x\cdot Dh(x)-2h(x)\right]\diff S_x\\
\geq&\exp\left\{\frac{1}{\omega_n r^{n-1}}\int_{\partial B_r(0)}\left[x\cdot Dh(x)-2h(x)\right]\diff S_x\right\}\\
=&\exp\left[rS_h'(r)-2S_h(r)\right].\nonumber
\end{align*}
Thus $S_h$ satisfies the following second order ordinary differential inequality
\begin{equation}\label{sm3}
S''_h(r)+\frac{n-1}{r}S_h'(r)\geq\exp\left[rS_h'(r)-2S_h(r)\right].
\end{equation}

Then we analyze above ordinary differential inequality. From \eqref{sm2} we see $S_h'(r)>0$ for $r>0$. 
Define an auxiliary function $p(r)$ by
\begin{equation*}
p(r)=rS_h'(r)-3S_h(r).
\end{equation*}
We claim $p(r)<0$ when $r\geq 4n$. Otherwise, there exists $r_1\geq 4n$ such that $p(r_1)\geq 0$. Define $r_2=\sup\{t\,|\,p(r)\geq 0,\,r\in[r_1,t]\}$. If $r_2<+\infty$, then
\begin{align*}
p'(r_2)&=r_2S_h''(r_2)-2S_h'(r_2) \\
&\geq r_2\left\{\exp\left[r_2S_h'(r_2)-2S_h(r_2)\right]-\frac{n-1}{r_2}S_h'(r_2)\right\}-2S_h'(r_2)\\
&=r_2\exp\left[\frac{1}{3}r_2S_h'(r_2)\right]-(n+1)S_h'(r_2)\\
&\geq\left(\frac{r_2^2}{3}-n-1\right)S_h'(r_2)>0.
\end{align*}
This contradicts the definition of $r_2$. Therefore $p(r)\geq 0$ holds on $[r_1, +\infty)$. Thus
\begin{align*}
S_h''(r)&\geq\exp\left[rS_h'(r)-2S_h(r)\right]-\frac{n-1}{r}S_h'(r)\\
&\geq\exp\left[\frac{1}{3}rS_h'(r)\right]-\frac{n-1}{r}S_h'(r)\\
&>\frac{1}{2}\exp S_h'(r).
\end{align*}
By Osgood's criterion, $S_h'(r)$ blows up in finite time, which contradicts the assumption that $h$ is entire. So the claim is true. For $r\geq 4n$ we have
\begin{equation}\label{sm4}
\frac{S_h'(r)}{S_h(r)}<\frac{3}{r}.
\end{equation}
Integrating \eqref{sm4}, we get
\begin{equation*}
S_h(r)<\frac{S_h\left(4n\right)}{64n^3}r^3\quad\,\mbox{for $r\geq 4n$.}
\end{equation*} 
Substituting above inequality into \eqref{sm4}, we obtain
\begin{equation}\label{sm5}
S_h'(r)<c_1r^2\quad\,\mbox{for $r\geq 4n$,}
\end{equation}
where $c_1=S_h\left(4n\right)$.

Now we have proved $S_h$ has at most a cubic growth. To get a finer estimate, we introduce another auxiliary function $q(r)$ given by
\begin{equation}
q(r)=rS_h'(r)-2S_h(r)-r. \nonumber
\end{equation}
We claim $q(r)<0$ when $r\geq n(c_1+4)$. The proof is similar. If the claim is not true, then there exists $r_3\geq n(c_1+4)$ such that $q(r_3)\geq 0$. Define $r_4=\sup\{t\,|\,q(r)\geq 0,\,r\in[r_3,t]\}$. If $r_4<+\infty$, then
\begin{align*}
q'(r_4)&=r_4S_h''(r_4)-S_h'(r_4)-1\\
&\geq r_4\left\{\exp\left[r_4S_h'(r_4)-2S_h(r_4)\right]-\frac{n-1}{r_4}S_h'(r_4)\right\}-S_h'(r_4)-1\\
&= r_4\exp r_4-nS_h'(r_4)-1\\
&\geq r_4\exp r_4-nc_1r_4^2-1>0.
\end{align*}
This contradicts the definition of $r_4$. Hence $rS_h'(r)-2S_h(r)\geq r$ holds on $[r_3, +\infty)$. It follows that
\begin{align*}
S_h''(r)&\geq\exp\left[rS_h'(r)-2S_h(r)\right]-\frac{n-1}{r}S_h'(r)\\
&>\exp r-nc_1r.
\end{align*}
Thus $S_h'(r)$ has an exponential growth as $r\rightarrow+\infty$, which contradicts \eqref{sm5}. 

Consequently, we have $rS_h'(r)-2S_h(r)<r$ for $r\geq n(c_1+4)$. Or equivalently,
\begin{equation}
\left[\frac{S_h(r)}{r^2}\right]'<\frac{1}{r^2}.\nonumber
\end{equation}
Integrating above inequality, we see when $r\geq 1$,
\begin{equation}\label{sm6}
\frac{S_h(r)}{r^2}<S_h\left(c_1n+4n\right)+1.
\end{equation}
Clearly $S_h$ has at most a quadratic growth. According to \eqref{sm2},
\begin{equation}\label{sm7}
\Phi_h(r)=\frac{nS_h'(r)}{r}.
\end{equation}
Combining \eqref{sm4}, \eqref{sm6} and \eqref{sm7}, we conclude that $\Phi_h(r)$ is bounded.
\end{proof}
\begin{lemma}\label{lem2}
Let $u$ be as stated in Theorem \ref{M}. Then $S_u(r)$ has at most a quadratic growth, and $\Phi_u(r)$ is bounded.
\end{lemma}
\begin{proof}
According to condition (ii) and \eqref{Meq}, we have
\begin{equation*}
C\left[\left(\Delta u\right)^{k_1}+1\right]\geq\exp\left(\frac{1}{2}x\cdot Du-u\right).
\end{equation*}
Since $u$ is strictly convex, $\Delta u>0$. If $k_1\geq 1$, we have
\begin{equation}\label{s1}
\left(\Delta u+1\right)^{k_1}\geq\left(\Delta u\right)^{k_1}+1\geq\exp\left(\frac{1}{2}x\cdot Du-u-\ln C\right).
\end{equation}
Set
\begin{equation}\label{lt}
\tilde u(x)=\frac{1}{2k_1}\left[u(x)+\frac{1}{2n}|x|^2+k_1\ln 2k_1+\ln C\right].
\end{equation}
Then it follows that
\begin{align*}
\Delta\tilde u(x)&=\frac{1}{2k_1}\left[\Delta u(x)+1\right]\\
&\geq\frac{1}{2k_1}\exp\frac{1}{2k_1}\left(x\cdot Du-2u-2\ln C\right)\\
&=\exp\left(x\cdot D\tilde u-2\tilde u\right).
\end{align*}
According to Lemma \ref{lem1}, $S_{\tilde u}(r)$ has at most a quadratic growth, and $\Phi_{\tilde u}(r)$ is bounded. Since we have the following relations
\begin{equation*}
S_u(r)=2k_1S_{\tilde u}(r)-\frac{1}{2n}r^2-k_1\ln 2k_1-\ln C
\end{equation*}
and
\begin{equation*}
\Phi_u(r)=2k_1\Phi_{\tilde u}(r)-1,
\end{equation*}
we conclude that $S_{u}(r)$ has at most a quadratic growth and $\Phi_{u}(r)$ is bounded.
For the case $k_1<1$, we have
\begin{equation}\label{s2}
2C\left(\Delta u+1\right)>C\left[\left(\Delta u\right)^{k_1}+1\right]\geq\exp\left(\frac{1}{2}x\cdot Du-u\right).
\end{equation}
In a very similar manner, we also draw the conclusion.
\end{proof}

For a convex function, once we know the growth of its spherical mean, we know the growth of itself as well as its gradient.
\begin{lemma}\label{lem3}
Let $h(x)$ be a $C^1$ convex function on $\mathbb{R}^n$. Assume that $S_h(r)$ has at most a quadratic growth. Then $h(x)$ has at most a quadratic growth, and $\left|Dh(x)\right|$ has at most a linear growth.
\end{lemma}
\begin{proof}
By the assumption, there exist positive constants $A$ and $B$ such that 
\begin{equation}\label{qg2}
S_h(r)\leq Ar^2+B\,\,\,\,\,\,\mbox{for all $r\geq 0$.}
\end{equation}
Since $h$ is convex, there exist positive constants $A'$ and $B'$ such that
\begin{equation}\label{qg3}
h(x)+A'|x|^2+B'\geq 0\,\,\,\,\,\,\mbox{for all $x\in\mathbb{R}^n$.}
\end{equation}
As $h(x)+A'|x|^2+B'$ is subharmonic, it satisfies mean value inequality. Then it follows that
\begin{equation}\label{ge4}
\begin{split}
h(x)+A'|x|^2+B'&\leq\frac{n}{\omega_{n} |x|^{n}}\int_{B_{|x|}(x)}\left[h(y)+A'|y|^2+B'\right]\diff y\\
&\leq\frac{n}{\omega_{n}|x|^{n}}\int_{B_{2|x|}(0)}\left[h(y)+A'|y|^2+B'\right]\diff y\\
&=\frac{n}{\omega_{n}|x|^{n}}\int_0^{2|x|}\int_{\partial B_t(0)}\left[h(z)+A'|z|^2+B'\right]\diff S_z\diff t\\
&\leq\frac{n}{\omega_{n}|x|^{n}}\int_0^{2|x|}\omega_{n}\left[(A+A')t^2+(B+B')\right]t^{n-1}\diff t\\
&\leq 2^n(A+A')|x|^2+2^n(B+B').
\end{split}
\end{equation}
The first inequality of \eqref{ge4} holds by the mean value inequality. The second one holds because of \eqref{qg3}. And the third one holds due to \eqref{qg2}.

Hence $h^+(x)=\max\{h(x),0\}$ has at most a quadratic growth. Since $h(x)$ is convex, $h^-(x)=\max\{-h(x),0\}$ has at most a linear growth. In conclusion, $h$ has at most a quadratic growth.

Attributable to the convexity of $h$, for an arbitrary unit vector $\xi\in\mathbb{R}^{n}$ we have
\begin{equation*}
\xi\cdot Dh\left(x\right)\leq\frac{h\left(x+|x|\xi\right)-h\left(x\right)}{|x|}.
\end{equation*}
This implies that $\left|Dh(x)\right|$ has at most a linear growth.
\end{proof}
Suppose that $u$ satisfies the conditions of Theorem \ref{M}. According to Lemma \ref{lem2}, $S_u(r)$ has at most a quadratic growth and $\Phi_u(r)$ is bounded. Then by Lemma \ref{lem3}, $u(x)$ has at most a quadratic growth and $\left|Du(x)\right|$ has at most a linear growth. The next lemma states that $u^-(x)$ grows sublinearly.
\begin{lemma}\label{lem4}
Let $h(x)$ be a $C^1$ convex function on $\mathbb{R}^n$. Suppose that $S_h(r)$ has at most a quadratic growth. And assume that for a certain positive constant $\alpha$, the ball mean of $\exp\alpha(x\cdot Dh-2h)$ is bounded. Then 
\begin{equation}
\lim_{|x|\rightarrow +\infty}\frac{h^-(x)}{\sqrt{|x|}}=0\nonumber.
\end{equation}
\end{lemma}
\begin{proof}
We proceed by contradiction. If the proposition is not true, then there exist a sequence $\{x_i\}_{i=1}^{\infty}\subset\mathbb{R}^n$ and a positive constant $c_2$ such that
\begin{equation}
h^-(x_i)\geq 3c_2|x_i|^{\frac{1}{2}},\,\,\,\,\,\,\mbox{and}\,\,\,\,\lim_{i\rightarrow+\infty}|x_i|=\infty. \nonumber
\end{equation}
According to Lemma \ref{lem3}, $|Dh|$ has at most a linear growth. Namely, there is a positive constant $c_3$ such that
\begin{equation}\label{nd0}
\left|D h(x)\right|\leq c_3\left|x\right|\quad\,\mbox{for $|x|\geq 1$.}
\end{equation}
Set $r_i=\frac{c_2}{c_3}\left|x_i\right|^{-\frac{1}{2}}$. Choose large enough $i$ for which $|x_i|\geq(\frac{c_2}{c_3})^2+2$. Then $r_i<1$. By \eqref{nd0} we have 
\begin{equation}\label{nd1}
h^-(x)\geq c_2\left|x_i\right|^{\frac{1}{2}}\,\,\,\,\,\mbox{for $x\in B_{r_i}\left(x_i\right)$.}
\end{equation}
It follows that
\begin{equation}\label{nd2}
\begin{split}
&\int_{B_{|x_i|+1}(0)}\exp\alpha\left[x\cdot Dh(x)-2h(x)\right]\diff x\\ \geq&\int_{B_{|x_i|+1}(0)}\exp\alpha\left[-h(x)-h(0)\right]\diff x\\
\geq&\exp\left[-\alpha h(0)\right]\int_{B_{r_i}(x_i)}\exp\left[\alpha h^-(x)\right]\diff x\\
\geq&\frac{\omega_{n}}{n}\left(\frac{c_2}{c_3}\right)^{n}\exp\left[-\alpha h(0)\right]\cdot\left|x_i\right|^{-\frac{n}{2}}\exp\left(c_2\alpha|x_i|^{\frac{1}{2}}\right).
\end{split}
\end{equation}
The first inequality of \eqref{nd2} holds due to the convexity of $h$. The second one holds because $B_{r_i}(x_i)\subset B_{|x_i|+1}(0)$. And the third one holds because of \eqref{nd1}. 

It follows that
\begin{align*}
&\lim_{i\rightarrow\infty}\frac{n}{\omega_n\left(\left|x_i\right|+1\right)^n}\int_{B_{|x_i|+1}(0)}\exp\alpha\left[x\cdot Dh(x)-2h(x)\right]\diff x\\
\geq&C\left(n,\alpha,h(0),c_2, c_3\right)\lim_{i\rightarrow\infty}\left|x_i\right|^{-\frac{3n}{2}}\exp\left(c_2\alpha\left|x_i\right|^{\frac{1}{2}}\right)=+\infty.
\end{align*}
This contradicts the assumption that the ball mean of $\exp\alpha(x\cdot Dh-2h)$ is bounded. So the proposition is true.\end{proof}

Because $\Phi_u(r)$ is bounded, from \eqref{s1} and \eqref{s2} we see that the ball mean of $\exp\alpha(x\cdot Du-2u)$ is bounded for $\alpha=\min\{1/2,1/2k_1\}$. Since $u$ is also convex, by Lemma \ref{lem4} we have
\begin{equation}
\lim_{|x|\rightarrow +\infty}\frac{u^-(x)}{\sqrt{|x|}}=0\nonumber.
\end{equation}

Having such estimates, we are in a position to construct a barrier function to prove the constancy of $\phi$.\\
\emph{Proof of Theorem \ref{HQ}.} Define the phase $\phi=\ln q_{n_1,n_2}\left(D^2u\right)$. By \eqref{HQeq}, we have
\begin{equation}\label{phase1}
\phi(x)=\frac{1}{2}x\cdot Du(x)-u(x).
\end{equation}
Taking two derivatives of \eqref{phase1}, we obtain
\begin{equation}
\phi_{ij}=\frac{1}{2}x^su_{ijs}.\label{diffphase1}
\end{equation}
Define the coefficients $a^{ij}\left(D^2u\right)$ by 
\begin{equation}
a^{ij}\left(D^2u\right)=\frac{\partial\ln q_{n_1,n_2}\left(D^2u\right)}{\partial u_{ij}}.\nonumber
\end{equation}
As shown above, $\left(a^{ij}\right)$ is positive-definite. A differentiation of \eqref{HQeq} with respect to $x^s$ yields
\begin{equation}\label{diffu1}
a^{ij}u_{ijs}=\phi_{s}.
\end{equation}
Combing \eqref{diffphase1} and \eqref{diffu1}, we get
\begin{equation}
a^{ij}\phi_{ij}-\frac{1}{2}x\cdot D\phi=0.\label{Lphase1}
\end{equation}
Thus $\phi$ satisfies an elliptic equation without zeroth order term (cf. \cite{CCY,H}). 

Define the corresponding elliptic operator by
\begin{equation*}
\mathcal{L}:=a^{ij}\partial^2_{ij}-\frac{1}{2}x\cdot D.
\end{equation*}
By \eqref{HQ2}, we have
\begin{equation}
a^{ij}u_{ij}=n_1-n_2.\nonumber
\end{equation}
For simplicity, denote $n_1-n_2$ by $N$. It follows that
\begin{equation*}
\mathcal{L}u= N-\frac{1}{2}x\cdot Du.
\end{equation*}
Define $\hat u(x)=u(x)-Du(0)\cdot x$.
Since $u$ is strictly convex, $\hat u$ is proper. And we have 
\begin{equation}\label{Lu1}
\mathcal{L}\hat u=N-\frac{1}{2}x\cdot D\hat u.
\end{equation}

Set $M=|u(0)|+1$. Define $l(x)$ by
\begin{equation}
l(x)=x\cdot Du(x)-u(x)+M. \nonumber
\end{equation}
Note that $l(x)\geq1$, and 
\begin{equation}
\mathcal{L}l=2\mathcal{L}\phi+\mathcal{L}u=N-\frac{1}{2}x\cdot Du.\nonumber
\end{equation}
Define $g(x)$ by
\begin{equation}
g(x)=\ln\left[\hat u(x)+M\right].\nonumber
\end{equation}
Note that $g(x)\geq 0$, and
\begin{equation}
\mathcal{L}g=\frac{\mathcal{L}\hat u}{\hat u+M}-\frac{a^{ij}\hat u_i\hat u_j}{\left(\hat u+M\right)^2}\leq\frac{1}{\hat  u+M}\left(N-\frac{1}{2}x\cdot D\hat u\right). \nonumber
\end{equation}
Then there holds
\begin{equation}\label{Llg1}
\begin{split}
\mathcal{L}\left(lg\right)=&l\mathcal{L}g+2a^{ij}g_il_j+g\mathcal{L}l\\
\leq&\frac{l}{\hat u+M}\left(N-\frac{1}{2}x\cdot D\hat u\right)+\frac{2\hat u_ia^{ij}u_{js}x^s}{\hat u+M}+\left(N-\frac{1}{2}x\cdot Du\right)\cdot\ln\left(\hat u+M\right).
\end{split}
\end{equation}
Denote the three terms on the right-hand side of \eqref{Llg1} by $I_1$, $I_2$ and $I_3$ respectively.

As talked above, $u$ has at most a quadratic growth and $|Du|$ has at most a linear growth. As well, $\hat u$ has at most a quadratic growth and $|D\hat u|$ has at most a linear growth. More precisely, there exists a positive constant $K_1$ such that
\begin{equation}\label{e11}
\hat u(x)+M\leq K_1|x|^2\,\,\,\,\,\mbox{for $|x|\geq 1$,}
\end{equation}
\begin{equation}\label{e21}
|D\hat u(x)|\leq K_1|x|\,\,\,\,\,\mbox{for $|x|\geq 1$}. 
\end{equation}
Attributable to the convexity and properness of $\hat u$, there is a positive constant $K_2$ such that
\begin{equation}\label{e31}
\hat u(x)+M\geq K_2|x|\,\,\,\,\,\mbox{for $|x|\geq 1$},
\end{equation}
\begin{equation}\label{e41}
x\cdot D\hat u(x)\geq K_2|x|\,\,\,\,\,\mbox{for $|x|\geq 1$}.
\end{equation}
As shown in Lemma \ref{lem4}, $u^-(x)=o(|x|^{\frac{1}{2}})$ as $x\rightarrow\infty$. Namely there exists a positive constant $K_3$ such that
\begin{equation}\label{e51‘}
u(0)-u(x)\leq K_3|x|^{\frac{1}{2}}\,\,\,\,\,\mbox{for $|x|\geq 1$}.
\end{equation}
By convexity, $-x\cdot Du(x)\leq u(0)-u(x)$. Thus we get
\begin{equation}\label{e51}
-x\cdot Du(x)\leq K_3|x|^{\frac{1}{2}}\,\,\,\,\,\mbox{for $|x|\geq 1$}.
\end{equation}
Since $l(x)\geq 1$, from \eqref{e41} we see when $|x|$ is large enough, 
\begin{equation}\label{e61}
I_1=\frac{l}{\hat u+M}\left(N-\frac{1}{2}x\cdot D\hat u\right)\leq 0.
\end{equation}
Define $E(x)=\hat u_i\left(x\right)a^{ij}\left(x\right)u_{js}\left(x\right)x^s$. By \eqref{HQ3} we have
\begin{equation}\label{e01}
E(x)\leq\tr\left(a^{ij}u_{js}\right)\cdot|x|\cdot\left|D\hat u\right|=N|x|\left|D\hat u\left(x\right)\right|.
\end{equation}
Then it follows from \eqref{e01}, \eqref{e21} and \eqref{e31} that for $|x|\geq 1$,
\begin{equation}\label{e71}
I_2=\frac{2E(x)}{\hat u+M}\leq 2NK_1K^{-1}_2|x|.
\end{equation}
According to \eqref{e11} and \eqref{e51}, when $|x|\geq 1$ we have
\begin{equation}\label{e81}
I_3=\left(N-\frac{1}{2}x\cdot Du\right)\cdot\ln\left(\hat u+M\right)\leq\left(\frac{K_3}{2}|x|^{\frac{1}{2}}+N\right)\left(2\ln|x|+\ln K_1\right).
\end{equation}
Substituting \eqref{e61}, \eqref{e71} and \eqref{e81} into \eqref{Llg1}, for large enough $|x|$, we have 
\begin{equation}
\mathcal{L}\left(lg\right)\leq 2NK_1K^{-1}_2|x|+K_3|x|^{\frac{1}{2}}(\ln|x|+\ln K_1)+2N\left(\ln|x|+\ln K_1\right).\nonumber
\end{equation}
Equations \eqref{Lu1} and \eqref{e41} then imply there exist $R_0\geq1$ and a large enough positive constant $K_4$ such that
\begin{equation}
\mathcal{L}\left(lg+K_4\hat u\right)\leq 0\,\,\,\,\,\mbox{when $|x|\geq R_0$}.\nonumber
\end{equation}

For any $\varepsilon>0$, we take a barrier function $\displaystyle w(x)$ defined by
\begin{equation}
w(x)=\varepsilon\left\{l\left(x\right)g\left(x\right)+K_4\left[\hat u(x)+M\right]\right\}+\max_{\partial B_{R_0}}\phi. \nonumber
\end{equation}
Clearly we have
\begin{equation}
\mathcal{L}w\leq 0=\mathcal{L}\phi\,\,\,\,\,\mbox{for $|x|\geq R_0$},\nonumber
\end{equation}
and
\begin{equation}
w(x)\geq\phi(x)\,\,\,\,\,\mbox{on $\partial B_{R_0}$}.\nonumber
\end{equation}
The last thing to check is
\begin{equation}
w(x)>\phi(x)\,\,\,\,\,\mbox{as $|x|\rightarrow+\infty$}.\nonumber
\end{equation}

We claim that above inequality holds when
\begin{equation}
|x|\geq\frac{1}{K_2}\exp{\frac{1}{\varepsilon}}+\left(\frac{2K_3}{K_2K_4\varepsilon}\right)^2+\frac{2}{K_2K_4\varepsilon}\left|\max_{\partial B_{R_0}}\phi\right|+R_0. \nonumber
\end{equation}
By \eqref{e31} we have
\begin{equation}\label{l11}
\frac{\varepsilon K_4}{2}\left[\hat u(x)+M\right]>\left|\max_{\partial B_{R_0}}\phi\right|,
\end{equation}
and
\begin{equation}\label{l31}
\varepsilon g(x)>1.
\end{equation}
Simple calculation yields
\begin{equation}\label{l21}
\frac{\varepsilon}{2} K_2K_4|x|>K_3|x|^{\frac{1}{2}}.
\end{equation}
Next we discuss the following two cases.

Case 1. $\phi(x)<\frac{\varepsilon}{2}K_2K_4|x|$. Directly from \eqref{e31} and \eqref{l11} we see
\begin{equation}
\phi (x)<\frac{\varepsilon K_4}{2}\left[\hat u(x)+M\right]\leq w(x). \nonumber
\end{equation}

Case 2. $\phi(x)\geq\frac{\varepsilon}{2}K_2K_4|x|$. By convexity, \eqref{e51‘} and \eqref{l21} we get
\begin{equation}
\begin{split}
\frac{1}{2}x\cdot Du&\geq\frac{\varepsilon}{2} K_2K_4|x|+u(x)\\
&\geq\frac{\varepsilon}{2}K_2K_4|x|-K_3|x|^{\frac{1}{2}}+u(0)\\
&\geq u(0).\nonumber
\end{split}
\end{equation}
Thus
\begin{equation}\label{l41}
l(x)-\phi(x)=\frac{1}{2}x\cdot Du+M\geq u(0)+M>0.
\end{equation}
Combing \eqref{l11}, \eqref{l31} and \eqref{l41}, we also have $w(x)>\phi(x)$.

The weak maximum principle then implies 
\begin{equation}
\varepsilon\left\{l\left(x\right)g\left(x\right)+K_4\left[\hat u(x)+M\right]\right\}+\max\limits_{\partial B_{R_0}}\phi\geq\phi(x)\,\,\,\,\,\mbox{for all $x\in\mathbb{R}^{n}\backslash B_{R_0}$}.\nonumber
\end{equation}
Letting $\varepsilon\rightarrow 0$, we obtain
\begin{equation}
\max\limits_{\partial B_{R_0}}\phi\geq\phi(x)\,\,\,\,\,\mbox{for all $x\in\mathbb{R}^{n}\backslash B_{R_0}$}.\nonumber
\end{equation}
So $\phi$ attains its global maximum in the closure of $B_{R_0}$. Hence $\phi$ is a constant by the strong maximum principle. Using $\phi=\frac{1}{2}x\cdot Du-u$, we have 
\begin{equation}
\frac{1}{2}x\cdot D\left[u(x)+\phi(0)\right]=u(x)+\frac{1}{2}\phi(0).\nonumber
\end{equation}

Finally, it follows from Euler's homogeneous function theorem that smooth $u(x)+\phi(0)/2$ is a homogeneous order $2$ polynomial. 
\qed

\section{Proof of Theorem \ref{KR1} and Theorem \ref{KR2}}
The whole proof of Theorem \ref{HQ} can be copied here except inequality \eqref{e01}. Actually, we only need to prove a \eqref{e01}-type inequality under the new conditions. For convenience and clarity, for the corresponding objects we use the same notations as in the proof of Theorem \ref{HQ}. 
\begin{proof}
When the eigenvalues of $\partial\bar\partial u$ are comparable, namely inequality \eqref{KRc1} holds, for any $i, s$ we have
\begin{equation*}
\left|\sum_{i=1}^{2n}a^{ij}(x)u_{js}(x)\right|<\frac{\Delta u(x)}{4\mu_{\min}(x)}\leq\frac{n\mu_{\max}(x)}{\mu_{\min}(x)}\leq n\Lambda.
\end{equation*}
So $E(x)\leq n\Lambda\left|x\right|\left|D\hat u(x)\right|$, where $\hat u(x)=u(x)-Du(0)\cdot x$.

Now we talk about the toric case. Since $u$ is invariant under $\mathbb{T}^n$-actions, we have $Du(0)=0$, $\hat u(x)$=$u(x)$. And $u(x)$ can be reduced to a function $f(r^1,...,r^n)$ depending only on each polar radius $r^i=\left|x^i+\sqrt{-1}x^{n+i}\right|$. Simple calculation gives:
\begin{equation*}
u_{i}=f_i\cdot\frac{x^i}{r^i},\,\,\,\,\,u_{n+i}=f_i\cdot\frac{x^{n+i}}{r^i}
\end{equation*}
for $1\leq i\leq n$, and
\begin{equation*}
u_{ij}=f_{ij}\cdot\frac{x^ix^j}{r^ir^j}+f_i\cdot\delta_{ij}\cdot\frac{(x^{n+i})^2}{(r^i)^3},
\end{equation*}
\begin{equation*}
u_{i,n+j}=f_{ij}\cdot\frac{x^ix^{n+j}}{r^ir^j}-\delta_{ij}\cdot f_i\cdot\frac{x^ix^{n+j}}{(r^i)^3},
\end{equation*}
\begin{equation*}
u_{n+i,n+j}=f_{ij}\cdot\frac{x^{n+i}x^{n+j}}{r^ir^j}+f_i\cdot\delta_{ij}\cdot\frac{(x^i)^2}{(r^i)^3}
\end{equation*}
for $1\leq i,j\leq n$.
So at $x=\left(r,0\right)$ where $r=(r^1,...,r^n)$, we have $Du=\left(Df\left(r\right),0\right)$ and
\begin{align*}
D^2u=\begin{pmatrix}
\displaystyle D^2f\left(r\right) & 0\\
0 & \Omega\left(r\right)
\end{pmatrix},
\end{align*}
where $\Omega\left(r\right)=\diag\left(\frac{f_1}{r^1},...,\frac{f_n}{r^n}\right)$.

As noted in the introduction, $\left(a^{ij}\right)=\left(D^2u-J\cdot D^2u\cdot J\right)^{-1}$.
$E(x)$ can be viewed as the matrix product 
\begin{equation*}
E(x)=\left(Du\right)^{\mathrm{T}}\cdot\left(D^2u-J\cdot D^2u\cdot J\right)^{-1}\cdot D^2u\cdot x,
\end{equation*}
where $\left(Du\right)^{\mathrm{T}}$ is the transpose of $Du$.

Since $u$ is $\mathbb{T}^n$-invariant and $J$ is an infinitesimal generator of $\mathbb{T}^n$-actions, $E(x)$ is $\mathbb{T}^n$-invariant. So $E(x)=E\left(r,0\right)$. Then it follows that
\begin{align*}
E\left(r,0\right)=&(Df)^{\mathrm{T}}\cdot\left(D^2f+\Omega\right)^{-1}\cdot D^2f\cdot r\\
=&(Df)^{\mathrm{T}}\cdot r-(Df)^{\mathrm{T}}\cdot\left(D^2f+\Omega\right)^{-1}\cdot\Omega\cdot r\\
=&(Df)^{\mathrm{T}}\cdot r-(Df)^{\mathrm{T}}\cdot\left(D^2f+\Omega\right)^{-1}\cdot Df\\
\leq &\left(Df\right)^{\mathrm{T}}\cdot r.
\end{align*}
Consequently, we have $E(x)\leq x\cdot Du(x)$. 
\end{proof}

\textbf{Aknowledgement}

I would like to sincerely thank Professor Yu Yuan for suggesting this problem to me and for many stimulating discussions. I am grateful to my advisor Professor Yuguang Shi for encouragement and useful advices. Most of this work was done when I was visiting the University of Washington. I would also like to thank CSC (China Scholarship Council) for its support and the University of Washington for its hospitality. Finally, I express my gratitude to the referees for pointing out some typos and for many useful  comments and suggestions.

\textbf{Funding}

This work was partially supported by the China Scholarship Council [201406010009] and the National Natural Science Foundation of China [11671015].



\end{document}